\documentclass[a4paper,12pt]{article}

\usepackage[centertags]{amsmath}
\usepackage{amsfonts}
\usepackage{amssymb}
\usepackage{amsthm}
\usepackage{graphicx}

\newtheorem{theorem}{Theorem}[section]
\newtheorem{lemma}[theorem]{Lemma}
\newtheorem{cor}[theorem]{Corollary}
\newtheorem{prop}[theorem]{Proposition}
\theoremstyle{definition}
\newtheorem{example}[theorem]{Example}

\allowdisplaybreaks[2]

\DeclareMathOperator{\dimH}{dim_H}

\begin{document}
\title{On the Hausdorff Dimension of Bernoulli Convolutions}
\author{Shigeki Akiyama, De-Jun Feng,\\ Tom Kempton and Tomas Persson}
\maketitle

\begin{abstract}
  \noindent We give an expression for the Garsia entropy of Bernoulli
  convolutions in terms of products of matrices. This gives an
  explicit rate of convergence of the Garsia entropy and shows that
  one can calculate the Hausdorff dimension of the Bernoulli
  convolution $\nu_{\beta}$ to arbitrary given accuracy whenever
  $\beta$ is algebraic. In particular, if the Garsia entropy
  $H(\beta)$ is not equal to $\log(\beta)$ then we have a finite time
  algorithm to determine whether or not $\dimH (\nu_{\beta})=1$.
\end{abstract}

\section{Introduction}

Bernoulli convolutions are a simple and interesting family of
self-similar measures with overlaps. For $\beta\in(1,2)$, the
Bernoulli convolution $\nu_{\beta}$ is defined as the weak-star limit
of the family of measures $\nu_{\beta}^{(n)}$ given by
\[
\nu_{\beta}^{(n)}:=\frac{1}{2^n}\sum_{a_1\cdots a_n\in\{0,1\}^n}
\delta_{\sum_{i=1}^n a_i\beta^{-i}}.
\]

The fundamental questions relating to the Bernoulli convolution
$\nu_{\beta}$ are whe\-th\-er it has Hausdorff dimension one, and if
so, whether it is absolutely continuous.

Erd\H{o}s proved that $\nu_{\beta}$ is singular whenever $\beta$ is a
Pisot number \cite{ErdosPisot}, and it was later proved by Garsia that
in fact $\nu_{\beta}$ has Hausdorff dimension less than one whenever
$\beta$ is Pisot \cite{GarsiaSingular}. So far, Pisot numbers are the
only class of $\beta$ for which it is known that $\nu_{\beta}$ is
singular. Garsia gave a small explicit class of $\beta$ for which
$\nu_{\beta}$ is absolutely continuous \cite{GarsiaAC}, until recently
these were the only examples of Bernoulli convolutions for which it
was known that the Hausdorff dimension is one. In \cite{SolomyakAC}
Solomyak proved that $\nu_{\beta}$ is absolutely continuous for
Lebesgue-almost every $\beta\in(1,2)$.

A great deal of the recent progress on Bernoulli convolutions stems
from Hochman's article \cite{HochmanInverse}, where it was proved that
if $\nu_{\beta}$ has Hausdorff dimension less than one then the sums
in the definition of $\nu_{\beta}^{(n)}$ must be superexponentially
close. This can only happen on a set of $\beta$ of Hausdorff dimension
zero. Additionally, Hochman proved that if $\beta$ is algebraic then
$\dimH(\nu_{\beta})$ can be expressed in terms of the Garsia entropy
of $\beta$, which will be defined in Section~\ref{Sec1.2}.

Further recent progress was made by Breuillard and Varju
\cite{VarjuBreuillard1}, where it was proved that
\[
H(\beta)\geq 0.44 \min \{\log 2, \log M_{\beta}\},
\]
for any algebraic integer $\beta\in (1,2)$, where $H(\beta)$ is the
Garsia entropy of $\nu_\beta$ (see Section~\ref{Sec1.2} for the
definition) and $M_\beta$ is the Mahler measure of $\beta$ defined by
$M_\beta=\prod_{|\beta_i|>1}|\beta_i|$, where $\beta_i$ are the
algebraic conjugates (including $\beta$ itself) of $\beta$. This
implies that for an algebraic integer $\beta\in (1,2)$,
$\dimH(\nu_{\beta})=1$ if $0.44 \min \{\log 2, \log M_{\beta}\}\geq
\log \beta$ (see \eqref{e-hochman}).

In \cite{VarjuBreuillard2}, Breuillard and Varju showed, among other
results, that
\[
\{\beta\in(1,2) : \dimH(\nu_{\beta}) < 1 \} \subset \overline{ \{
  \beta \in (1,2) \cap \overline{\mathbb Q}:\dimH(\nu_{\beta})<1\}},
\]
where $\overline{\mathbb Q}$ is the set of algebraic numbers. This,
together with Hochman's results, has sparked renewed interest in the
study of Garsia entropy for algebraic parameters. If one were able to
show that Pisot numbers are the only algebraic numbers corresponding
to Bernoulli convolutions of dimension less than one, this would show
that all non-Pisot $\beta$ give rise to Bernoulli convolutions of
dimension $1$ (without the restriction that $\beta$ should be
algebraic).

There have also been recent results on the absolute continuity of
Bernoulli convolutions. Shmerkin \cite{ShmerkinAC} proved further that
$\nu_{\beta}$ is absolutely continuous for all $\beta\in(1,2)\setminus
\mathcal E$ where $\mathcal E$ is a set of exceptions of Hausdorff
dimension zero. In \cite{VarjuAC}, Varju gave new explicit examples of
absolutely continuous Bernoulli convolutions. For a recent summary of
progress on Bernoulli convolutions, see \cite{VarjuSummary}.

In this article we are interested in expressing the Garsia entropy and
the dimension of Bernoulli convolutions $\nu_{\beta}$ in terms of
products of matrices. There is some precedent for this, see in
particular \cite{FengSidorov}, but these previous ideas are based on
tilings of the unit interval related to the construction of
$\nu_{\beta}$, and cannot be generalised to the non-Pisot cases. We
use a different approach to show that, for any algebraic integer
$\beta$, one can construct matrices whose products encode information
about the Garsia entropy. In particular, we give a sequence of lower
bounds for the Garsia entropy which yield an explicit rate of
convergence in the Garsia entropy formula.

\subsection{Statement of Results}\label{Sec1.2}

Let $\Sigma:=\{0,1\}^{\mathbb N}$. For $p\in(0,1)$, let $m_p$ denote
the $(p,1-p)$ Bernoulli product measure on $\Sigma$ which gives weight
$p$ to digit $0$ and weight $1-p$ to digit $1$. For $\beta \in (1,2)$,
the transformation $\pi_\beta \colon \Sigma \to \mathbb{R}$ defined by
\[
\pi_\beta \colon (a_i)_{i=1}^\infty \mapsto \sum_{i=1}^\infty a_i
\beta^{-i},
\]
maps the measure $m_p$ to a measure $\nu_{\beta,p}$ on
$\mathbb{R}$. That is, $\nu_{\beta,p}=m_{\beta}\circ
\pi^{-1}_\beta$. For $p=\frac{1}{2}$, we get the Bernoulli convolution
$ \nu_\beta=\nu_{\beta,\frac{1}{2}} $, which was defined in the
previous section. For $p \neq \frac{1}{2}$ we get a so-called biased
Bernoulli convolution.

Given a word $a_1\cdots a_n\in\{0,1\}^n$, let the cylinder set
$[a_1\cdots a_n]$ be defined by
\[
[a_1\cdots a_n] = \{\underline b=(b_i)_{i=1}^{\infty} \in \Sigma :
b_1\cdots b_n = a_1\cdots a_n\}.
\]

Given a sequence $\underline a=(a_i)_{i=1}^{\infty}\in\{0,1\}^{\mathbb
  N}$, let
\[
\mathcal N_n(\underline a)=\mathcal N_n(a_1,\ldots,
a_n)=\left| \left\{(b_1,\ldots,b_n)\in\{0,1\}^n: \sum_{i=1}^n
b_i\beta^{-i}=\sum_{i=1}^n a_i\beta^{-i}\right\}\right|
\]
and
\[
\mathcal M_n(\underline a,p) = \sum_{\substack{b_1 \cdots b_n \in
    \{0,1\}^n \\ \sum_{i = 1}^n b_i \beta^{-i} = \sum_{i=1}^n a_i
    \beta^{-i}}} m_p [b_1 \cdots b_n].
\]
In what follows we write $\mathcal M_n(\underline a)$ or $\mathcal
M_n(a_1\cdots a_n)$, de-emphacising the dependence on $p$ since we
consider $p$ to be fixed.

%Further, let
%\[
%p_n(\underline a)=\dfrac{\mathcal N_n(\underline a)}{2^n}
%\]
Let\footnote{ It would be more standard to write \[
  H_n(\beta)=\sum_{x}\mathcal M_n(x)\log \mathcal M_n(x),
\]
where the sum is over all $x$ having a representation $x=\sum_{i=1}^n
a_i\beta^{-i}$ and $\mathcal M_n(x)$ is just $\mathcal M_n(\underline
a)$ for any $\underline a$ with $x=\sum_{i=1}^n a_i\beta^{-i}$. These
expressions are clearly equivalent, we find ours more convenient since
we work only with sequences and since the above makes the link with
Lyapunov exponents of pairs of matrices more direct.}
\[
H_n(\beta,p):= -\sum_{a_1\cdots a_n\in\{0,1\}^n} m_p[a_1\cdots
  a_n]\log \mathcal M_n(a_1 \cdots a_n).
\]
Finally we let
\[
H(\beta,p):=\lim_{n\to\infty}\frac{1}{n} H_n(\beta,p).
\]
$H(\beta,p)$ is called the Garsia entropy \footnote{Beware, there are
  two different conventions for the definition of Garsia entropy. Some
  authors divide by $\log(\beta)$ in the definition.}  of
$\nu_{\beta,p}$.  In particular, we write $H_n(\beta)=H_n(\beta,1/2)$
and $H(\beta)=H(\beta,1/2)$.

Hochman \cite{HochmanInverse} proved that if $\beta\in(1,2)$ is
algebraic then the dimension of the Bernoulli convolution
$\nu_{\beta,p}$ is given by
\begin{equation}
  \label{e-hochman}
  \dimH(\nu_{\beta,p}) = \min \biggl\{\dfrac{H(\beta,p)}{\log
    \beta},1\biggr\};
\end{equation}
see also \cite{VarjuBreuillard1} for a more detailed explanation.

In this article we are concerned with lower bounds for $H(\beta,p)$,
and hence lower bounds for $\dimH(\nu_{\beta,p})$, when $\beta$ is
algebraic. If $\beta$ is not an algebraic integer, i.e.\ not the root
of a polynomial with integer coefficients where the leading
coefficient is $1$, then $H(\beta,p)=\log 2$. Thus we may restrict our
interest to algebraic integers.

Given an algebraic integer $\beta=\beta^{(1)}$ of degree $d$, let
$\beta^{(2)},\ldots,\beta^{(d)}$ denote its Galois conjugates, ordered
by decreasing absolute value.

\begin{theorem}\label{thm1}
  Let $\beta$ be an algebraic integer of degree $d$ and let
  $p\in(0,1)$. The Garsia entropy $H(\beta,p)$ can be approximated
  with explicit error bounds. In particular,
  \[
  \frac{1}{n}H_n(\beta,p)-\frac{C+l\log (n+1)}{n}\leq H(\beta,p)\leq
  \frac{1}{n}H_n(\beta,p)
  \]
  for all $n\in {\mathbb N}$, where
  \[
  C=\log\left(2^d \prod_{i:|\beta^{(i)}|\neq 1}
  \frac{1}{||\beta^{(i)}|-1|} + 1 \right),
  \]
and $l$ is the number of conjugates of $\beta$ of absolute value $1$.
\end{theorem}

Theorem~\ref{thm1} is proved by giving lower bounds for $H(\beta,p)$
in terms of products of matrices.

\begin{theorem}\label{thm2}
  There exists a pair of matrices $M_0=M_0(\beta,p)$ and
  $M_1=M_1(\beta,p)$, with rows and columns indexed by a set $\mathcal
  A$, such that the sequence
  \[
  \frac{1}{n}L_n(\beta,p):= -\frac{1}{n} \sup_{i\in \mathcal A}
  \sum_{a_1 \cdots a_n \in \{0,1\}^n} m_p [a_1\cdots a_n] \log \Biggl(
  \sum_{j \in \mathcal A} (M_{a_1} \cdots M_{a_n})_{i,j} \Biggr)
  \]
  converges to $H(\beta,p)$ from below as $n\to\infty$, and
  $\frac{1}{n} L_n (\beta, p) \leq H(\beta, p)$.
\end{theorem}

The set $\mathcal A$ is finite (with size bounded by $C(\beta)$ given
by \eqref{eq:C(beta)definition}) whenever $\beta$ is hyperbolic,
i.e.\ when it has no Galois conjugates of modulus one. In this case
the matrices $M_0, M_1$ are computable by a finite time algorithm. If
$\beta$ is not hyperbolic then $\mathcal A$ might be countably
infinite, but the matrices $M_0, M_1$ have at most two non-zero terms
in any row.

Theorems~\ref{thm1} and \ref{thm2} are proved by bounding the
difference between $H_n(\beta,p)$ and $L_n(\beta,p)$. When $\beta$ is
hyperbolic, and so $\mathcal A$ is finite, Theorem~\ref{thm2} can be
expressed in the more familiar form of the Lyapunov exponent of the
pair of matrices $M_0, M_1$.

\begin{theorem}\label{thm3}
  When $\beta$ is hyperbolic, the sequence
  \[
  \frac{1}{n}L_n'(\beta,p):=-\frac{1}{n}\sum_{a_1\cdots
    a_n\in\{0,1\}^n}p(a_1\cdots a_n)\log(\lVert M_{a_1}\cdots
  M_{a_n}\rVert)
  \]
  converges to $H(\beta,p)$ as $n\to\infty$, and $\frac{1}{n} L_n'
  (\beta, p) \leq \frac{1}{n} L_n (\beta, p) \leq H(\beta, p)$.
\end{theorem}

%Here the matrices $M_0, M_1$ are finite, and computable by a finite
%time algorithm, whenever $\beta$ is hyperbolic (i.e.\ whenever $\beta$
%has no conjugates of modulus one). If $\beta$ is not hyperbolic then
%the matrices $M_0, M_1$ are infinite, but have at most two non-zero
%entries in any row, and there is a finite time algorithm to compute
%$L_n(\beta,p)$. The norm that we use is the row-sum norm.

An immediate corollary is that we can express the Garsia entropy as
the Lyapunov exponent of the matrices $M_0, M_1$ associated with the
$(p, 1-p)$-Bernoulli product measure.

\begin{cor}\label{cor1}
  If $\beta$ is hyperbolic then the Garsia entropy $H(\beta,p)$ is the
  limit of the sequence
  \[
  -\frac{1}{n}\log \lVert M_{a_1}\cdots M_{a_n} \rVert
  \]
  for $m_p$-a.e.\ $\underline a \in\{0,1\}^{\mathbb N}.$
\end{cor}

That corollary~\ref{cor1} follows from Theorem~\ref{thm3} is an
immediate application of the main result of \cite{FurstenbergKesten}.

\section{Preliminary Results}
In this section we recall some standard algebraic lemmas as well as
ideas about separation of polynomials originating in the work of
Garsia \cite{GarsiaAC}.

Let $\beta=\beta^{(1)}\in(1,2)$ be an algebraic integer of degree
$d$. Let $\beta^{(2)}, \ldots, \beta^{(r)}$ denote the algebraic
conjugates of $\beta$ of modulus strictly larger than one,
$\beta^{(r+1)}, \ldots, \beta^{(r+l)}$ conjugates of modulus $1$, and
$\beta^{(r+l+1)}, \ldots, \beta^{(d)}$ conjugates of modulus less than
one.

The following lemmas are standard.

\begin{lemma}\label{basic1}
  If $\sum_{i=1}^n \epsilon_i\beta^{-i}=0$ for
  $\epsilon_i\in\{-1,0,1\}$ then
  \[
  \sum_{i=1}^n \epsilon_i(\beta^{(j)})^{-i}=0
  \]
  for each $j\in\{2,\ldots,d\}$.
\end{lemma}

\begin{lemma}
  Let $P$ be a polynomial with integer coefficients. Then the product
  $P(\beta)P(\beta^{(2)})\cdots P(\beta^{(d)})$ is an integer.
\end{lemma}
Note that this second lemma requires that $\beta$ is an algebraic
integer, i.e. the root of a polynomial with integer coefficients whose
leading term is $1$. It does not hold for all algebraic numbers.

Define the set
$V_{\beta,n}\subset\left[\dfrac{-1}{\beta-1},\dfrac{1}{\beta-1}\right]$
by
\begin{multline*}
  V_{\beta,n}:=\biggl\{ x=\sum_{i=0}^n \epsilon_i \beta^{n-i}:
  \epsilon_i\in\{-1,0,1\} \text{ and }\\ \left|\sum_{i=0}^n
  \epsilon_i(\beta^{(j)})^{n-i}\right|\leq \frac{1}{|\beta^{(j)}|-1}
  \ \text{ for all } j\in \{1,\ldots r\} \biggr\}.
\end{multline*}
Let
\[V_{\beta}:=\bigcup_{n=0}^{\infty} V_{\beta,n}.
\]

\begin{lemma}
  \label{lem-2.3}
  Suppose that $\sum_{i=0}^n \epsilon_i\beta^{n-i}=0$. Then
  \[
  \sum_{i=0}^m \epsilon_i\beta^{m-i}\in V_{\beta}
  \]
  for each $m\in\{0,\ldots,n\}$.
\end{lemma}
%This relates to our definition of Weak Garsia Separation in that it says that $D'_{\beta}\subset V_{\beta}$.
\begin{proof}
  Suppose on the contrary that
  \[
  \sum_{i=0}^m\epsilon_i\beta^{m-i}\not\in V_\beta \qquad \text{for
    some } m\in \{0, 1,\ldots, n-1\}.
  \]
  Then by definition, there exists
  $j\in\{1,\ldots,r\}$ such that
  \[
  \left|\sum_{i=0}^m \epsilon_i(\beta^{(j)})^{m-i}\right|>
  \frac{1}{|\beta^{(j)}|-1}.
  \]
  Then
  \begin{align*}
    \left|\sum_{i=0}^{m+1} \epsilon_i(\beta^{(j)})^{m+1-i}\right| &=
    \left|\epsilon_{m+1}+\beta^{(j)} \sum_{i=0}^m \epsilon_i
    (\beta^{(j)})^{m-i}\right|\\ &\geq |\beta^{(j)}| \left|
    \sum_{i=0}^m \epsilon_i(\beta^{(j)})^{m-i}\right|-1\\ &\geq
    \dfrac{|\beta^{(j)}|}{|\beta^{(j)}|-1}-1 \geq
    \frac{1}{|\beta^{(j)}|-1}.
  \end{align*}
  Iterating this argument gives that
  \[
  \left|\sum_{i=0}^n
  \epsilon_i(\beta^{(j)})^{n-i}\right|>\frac{1}{|\beta^{(j)}|-1}.
  \]
  But by Lemma~\ref{basic1} the quantity on the left hand side is
  equal $0$, since $\sum_{i=0}^n \epsilon_i\beta^{n-i}=0$. This gives
  a contradiction.
\end{proof}

Let
\begin{equation} \label{eq:C(beta)definition}
  C(\beta) := 2^d \prod_{j=1}^r
  \frac{1}{|\beta^{(j)}|-1}\prod_{k=r+l+1}^{d}\frac{1}{1-|\beta^{(k)}|}
  = 2^d \prod_{|\beta^{(j)}| \neq 1} \frac{1}{||\beta^{(j)}|-1|}.
\end{equation}
This is a product over all roots which do not have modulus one. The
following lemma is essentially due to Garsia, see also
\cite{Frougny92, ATZ, Mercat2013}.

\begin{lemma}\label{separation}
  We have
  \[
  |V_{\beta,n}|\leq C (\beta) (n+1)^l + 1.
  \]
  In particular, if $\beta$ is hyperbolic then $V_{\beta}$ is finite.
\end{lemma}

\begin{proof}
  Let
  $V'_{\beta,n}\subset\left[\dfrac{-2}{\beta-1},\dfrac{2}{\beta-1}\right]$
  be given by
  \begin{multline*}
    V'_{\beta,n}:=\biggl\{x=\sum_{i=0}^n \epsilon_i \beta^{n-i}:
    \epsilon_i\in\{-2,-1,0,1,2\} \text{ and } \\ \left|\sum_{i=0}^n
    \epsilon_i(\beta^{(j)})^{n-i}\right|\leq \frac{2}{|\beta^{(j)}|-1}
    \ \text{ for all } j\in \{1,\ldots r\} \biggr\}.
  \end{multline*}
  For a non-zero $x\in V_{\beta,n}'$, given by
  \[
  x=\sum_{i=0}^n \epsilon_i \beta^{n-i}
  \]
  with $\epsilon_i\in\{-2,-1,0,1,2\}$, write
  \[
  x^{(j)}=\sum_{i=0}^n \epsilon_i (\beta^{(j)})^{n-i}.
  \]
  Then
  \begin{equation}\label{prod}
    \prod_{j=1}^d |x^{(j)}| \geq 1,
  \end{equation}
  since $\prod_{j=1}^d x^{(j)}$ is an integer, which is non-zero as $x
  \neq 0$.

  Now for $j\in\{r+l+1,\ldots,d\}$,
  \[
  |x^{(j)}| = \left|\sum_{i=0}^n \epsilon_i (\beta^{(j)})^{n-i}\right|
  \leq \sum_{i=0}^n 2 |\beta^{(j)}|^{n-i} \leq
  \frac{2}{1-|\beta^{(j)}|}.
  \]
  Furthermore, for $j\in\{2,\ldots,r\}$,
  \[
  |x^{(j)}|\leq \frac{2}{|\beta^{(j)}|-1},
  \]
  since $x\in V_{\beta,n}'$.

  Finally, for $j\in\{r+1,\cdots,r+l\}$,
  \[
  |x^{(j)}| = \left|\sum_{i=0}^n \epsilon_i
  (\beta^{(j)})^{n-i}\right| \leq \sum_{i=0}^n
  2 \cdot 1^{n-i}\leq 2 (n+1).
  \]
  Then by \eqref{prod},
  \[
    |x| \geq \frac{1}{\prod_{i\in\{2,\ldots,d\}}|x^{(i)}|} \geq C_0
    (n),
  \]
  where
  \[
  C_0 (n) := 2^{-(d-1)} \left(\prod_{j\in\{2,\ldots,r\}}
  (|\beta^{(j)}|-1) \right) \frac{1}{(n+1)^l} \left(
  \prod_{j\in\{r+1,\ldots,d\}} (1-|\beta^{(j)}|) \right).
  \]
  Hence, any $x\in V'_{\beta,n}\setminus \{0\}$ has modulus at least
  $C_0(n)$.

  Then for $y,z\in V_{\beta,n}$ with $y\neq z$, we have
  \[
  0 \neq y-z \in V'_{\beta,n}.
  \]
  Hence $|y-z|\geq C_0 (n)$. This shows that any two different
  elements of $V_{\beta,n}$ are separated by at least $C_0
  (n)$. Therefore, since $V_{\beta,n} \subset [-1/(\beta-1),
    1/(\beta-1)]$, this shows that $V_{\beta,n}$ contains at most
  \[
  \frac{2}{\beta - 1} \frac{1}{C_0 (n)} + 1 = C(\beta) (n+1)^l + 1
  \]
  elements.
\end{proof}

\section{Matrices, Lyapunov Exponents, and Lower Bounds for Garsia Entropy}
We now show how the sets $V_{\beta,n}$ of the previous section have a
natural graph structure, which allows one to compute lower bounds for
Garsia entropy.

Start with the sets $V_{\beta,0}=\{1,0,-1\}$ and
$A_{0}=\{1,0,-1\}$. At stage $n\geq 1$ we let
$V_{\beta,n}=V_{\beta,n-1}\bigcup A_n$ where
\[
A_n=\{\beta x -\epsilon_n: \epsilon_n\in\{-1,0,1\}, x\in A_{n-1},
\beta x - \epsilon_n \in V_{\beta}\}.
\]
If $\beta$ is hyperbolic, we stop the algorithm at the stage $n$ for
which $V_{\beta, n} = V_{\beta, n-1}$. Since in the hyperbolic case
$V_{\beta}$ is finite, the algorithm must stop in finite time with
$V_{\beta,n}=V_{\beta}$. If $\beta$ is not hyperbolic then $V_{\beta}$
may be countably infinite, but $V_{\beta,n}$ grows at most
polynomially in $n$.

For each $x,y\in V_{\beta}$, draw a directed edge from $x$ to $y$,
labelled by $\epsilon\in\{-1,0,1\}$, whenever $y=\beta x +
\epsilon$. Call the resulting graph $\mathcal G$.

%Finally we prune the graph. Call a vertex redundant if there is no
%path back to $\{0\}$ from $x$ along edges in the graph. Clearly, a
%vertex is redundant, only if all edges from the vertex lead to
%redundant vertices. We end up with a finite directed graph $\mathcal
%G$ whose vertices are precisely the elements of the set
%$D'_{\beta}$. This completes the proof of Theorem~\ref{thm1}.

There is a simple connection between the graph $\mathcal G$ and the
quantities $\mathcal N_n(\underline a)$.

Suppose that for $\underline a=(a_i)_{i=1}^{\infty}$ and $\underline
b=(b_i)_{i=1}^{\infty}$ we have
\[
\sum_{i=1}^n a_i\beta^{-i}=\sum_{i=1}^n b_i\beta^{-i}.
\]
Then, by the definition of $V_{\beta}$ and Lemma \ref{lem-2.3}, for
each $m\in\{1,\ldots,n\}$ we have
\[
d_m(\underline a,\underline b):=\beta^m\sum_{i=1}^m
(a_i-b_i)\beta^{-i}\in V_{\beta}.
\]
Then, letting $d_0(\underline a,\underline b):= 0$, we see that the
word $d_0(\underline a,\underline b)d_1(\underline a,\underline
b)\cdots d_n(\underline a,\underline b)$ follows a path from $0$ to
$0$ on the graph $\mathcal G$, following at each step $i$ an edge
labelled by $(a_i-b_i)\in\{-1,0,1\}$.

Given a word $a_1\cdots a_n\in\{0,1\}^n$ and
$\epsilon_1\cdots\epsilon_n\in\{-1,0,1\}^n$, we write
\[
 \epsilon_1\cdots \epsilon_n \sim a_1\cdots a_n
\]
if $a_i-\epsilon_i\in\{0,1\}$ for each $i\in\{1,\ldots,n\}$. Then
\begin{multline*}
  \mathcal N_n(\underline
  a)=\bigl|\bigl\{\epsilon_1\cdots\epsilon_n\sim a_1\cdots a_n\text{
    such that there is a path}\\ \text{ from $0$ to $0$ in $\mathcal
    G$ obtained by following the edges $\epsilon_1\cdots\epsilon_n$}
  \bigr\}\bigr|.
\end{multline*}

We can write down matrices which encode the choices of move
$\epsilon_i$ available given $a_i$.

Let $x_1,x_2,\ldots$ be some ordering of the elements of $V_{\beta}$,
with $x_1=0$. Let $\mathcal A=\{1,\ldots, |V_{\beta}|\}$ if
$V_{\beta}$ is finite, and $\mathbb N$ otherwise. We want
to write down matrices $M_0$ and $M_1$ such that, for a word
$a_1\cdots a_n$,
\[
(M_{a_1}\cdots M_{a_n})_{i,j} = \sum_{\substack{b_1\cdots b_n \in
    \{0,1\}^n \\ \beta^n x_i +\sum_{l=1}^n (a_l-b_l)\beta^{n-l}=x_j}}
m_p[b_1\cdots b_n].
\]
Let $M_0$ be the $|V_{\beta}|\times |V_{\beta}|$ matrix such that
\[
(M_0)_{i,j}=\left\lbrace \begin{array}{ll} 1-p & \text{if } x_j =
  \beta x_i - 1 \\ p & \text{if } x_j = \beta x_i\\ 0 & \text{
    otherwise} \end{array} \right. ,
\]
and let $M_1$ be the $|V_{\beta}|\times |V_{\beta}|$ matrix such that
\[
(M_1)_{i,j}=\left\lbrace \begin{array}{ll} 1-p & \text{if } x_j =
  \beta x_i \\ p & \text{if } x_j = \beta x_i+1\\ 0 & \text{
    otherwise} \end{array} \right. .
\]

\begin{lemma}\label{MatrixLemma}
  \begin{enumerate}
  \item For $x_i, x_j \in V_{\beta}$,
    \[
    (M_{a_1}\cdots M_{a_n})_{ij} = \sum_{\substack{b_1\cdots b_n \in
        \{0,1\}^n \\ \beta^n x_i + \sum_{l=1}^n
        (a_l-b_l)\beta^{n-l}=x_j }} m_p[b_1\cdots b_n].
    \]
  \item If $\sum_{i=1}^n a_i\beta^{-i}=\sum_{i=1}^n c_i\beta^{-i}$ for
    some $a_1\cdots a_n$, $c_1\cdots c_n\in \{0,1\}^n$, then
    \[
    M_{a_1}\cdots M_{a_n}=M_{c_1}\cdots M_{c_n}.
    \]
  \end{enumerate}
\end{lemma}

\begin{proof}
  Part 1 follows immediately from the definition of $M_0$ and
  $M_1$, and part 2 follows directly from part 1.
\end{proof}
In particular,
\[
\mathcal M_n(\underline a)=\left(M_{a_1}M_{a_2}\cdots
M_{a_n}\right)_{11},
\]
and so we have immediately that
\[
H_n(\beta, p)=-\sum_{a_1\cdots a_n}
m_p[a_1\cdots a_n]\log\left(\left(M_{a_1}M_{a_2}\cdots
  M_{a_n}\right)_{11}\right).
\]
As $H_n(\beta,p)$ is subadditive,  $\frac{1}{n}H_n(\beta,p)$ forms a
sequence which converges to $H(\beta, p)$. We have $H(\beta,p) \leq
\frac{1}{n} H_n (\beta,p)$.

\subsection{Lower Bounds}

For an algebraic integer $\beta$, we define
\begin{align*}
  L_n(\beta,p):&= - \sup_{i\in \mathcal A} \sum_{a_1 \cdots a_n \in
    \{0,1\}^n} m_p [a_1\cdots a_n] \log \Biggl( \sum_{j \in \mathcal A}
  (M_{a_1} \cdots M_{a_n})_{i,j} \Biggr)\\ &= - \sup_{i \in \mathcal A}
  \int_{\underline a \in \{0,1\}^n} \log \Biggl(\sum_{j \in \mathcal A}
  (M_{a_1} \cdots M_{a_n})_{i,j} \Biggr) dm (\underline a).
\end{align*}
Since
\[
M_{11}\leq \sum_{j\in V_{\beta}} M_{1j},
\]
we have, by choosing $i=1$ in the above definition, that $L_n
(\beta,p) \leq H_n (\beta,p)$. Here, and in much of what follows, we
note the minus in the definition of $H_n(\beta,p)$ and $L_n(\beta,p)$
which reverses a lot of inequalities.

\begin{lemma}
  \[
  L_{n+m}(\beta,p)\geq L_n(\beta,p)+L_m(\beta,p).
  \]
\end{lemma}

\begin{proof}
  For $i\in \mathcal A$, $\underline a\in\Sigma$ we have
  \begin{align*}
    \sum_{j\in \mathcal A}(M_{a_1}\cdots M_{a_{n+m}})_{i,j}&=
    \sum_{j\in \mathcal A}\sum_{k\in \mathcal A}(M_{a_1}\cdots
    M_{a_n})_{i,k}(M_{a_{n+1}}\cdots
    M_{a_{n+m}})_{k,j}\\ &=\sum_{k\in\mathcal A}(M_{a_1}\cdots
    M_{a_n})_{i,k}\left(\sum_{j\in \mathcal A} (M_{a_{n+1}}\cdots
    M_{a_{n+m}})_{k,j}\right)\\ &\leq \sum_{k\in\mathcal
      A}(M_{a_1}\cdots M_{a_n})_{i,k}\left(\sup_{l\in\mathcal
      A}\sum_{j\in \mathcal A} (M_{a_{n+1}}\cdots
    M_{a_{n+m}})_{l,j}\right)
  \end{align*}
  It follows that
  \begin{align*}
    L_{n+m}(\beta,p) =& - \sup_{i\in \mathcal A}\int_{\underline
      a\in\{0,1\}^{\mathbb N}}\log \Biggl( \sum_{j \in \mathcal A}
    (M_{a_1}\cdots M_{a_{n+m}})_{i,j} \Biggr) dm(\underline a)\\ \geq&
    -\sup_{i\in \mathcal A}\sup_{l\in \mathcal A}\int_{\underline a}
    \int_{\underline b} \Biggl( \log\left( \sum_{k\in \mathcal A}
    (M_{a_1}\cdots M_{a_n})_{i,k} \right)\\ & \hspace{3.2cm} +
    \log\left( \sum_{j\in \mathcal A} (M_{b_{1}}\cdots M_{b_m})_{l,j}
    \right) \Biggr) dm(\underline a)dm(\underline b)\\ =&
    L_n(\beta,p)+L_m(\beta,p). \qedhere
  \end{align*}
\end{proof}

\begin{prop} \label{prop:boundsonH}
  Let $\beta\in(1,2)$ be an algebraic integer. Then the sequence
  $\left(\frac{1}{n}L_n(\beta,p)\right)$ satisfies
  \[
  \frac{1}{n}H_n(\beta,p)-\frac{1}{n}\log(C(\beta)(n+1)^l+1) \leq
  \frac{1}{n}L_n(\beta,p) \leq H (\beta,p)\leq\frac{1}{n}H_n(\beta,p).
  \]
  \end{prop}

\begin{proof}
  We have proved that $L_n(\beta,p)$ is superadditive. Since $H_n
  (\beta,p)$ is subadditive, $\frac{1}{n}H_n(\beta, p)$ converges to
  $H(\beta,p)$ and
  $\frac{1}{n}H_n(\beta,p)\geq\frac{1}{n}L_n(\beta,p)$, we see that
  \[
  H(\beta,p) \in \Bigl( \frac{1}{n}L_n(\beta,p),
  \frac{1}{n}H_n(\beta,p) \Bigr)
  \]
  for all $n\in\mathbb N$. Hence we need only to prove the left hand
  inequality.

  Let
  \[
  X_n:=\left\lbrace\sum_{i=1}^n
  a_i\beta^{-i}:a_i\in\{0,1\}\right\rbrace.
  \]

  For $x\in X_n$ let $M_{x,n}:=M_{a_1}\cdots M_{a_n}$ for any of the
  words $a_1\cdots a_n$ for which
  \[
  x = \sum_{i=1}^n a_i \beta^{-i}.
  \]
  This is well defined due to Lemma~\ref{MatrixLemma}. Now
  \begin{align*}
    L_n(\beta,p):=& -\sup_{i\in \mathcal A}\sum_{a_1\cdots
      a_n}m_p[a_1\cdots a_n]\log\left(\sum_{j\in \mathcal A}
    (M_{a_1}\cdots M_{a_n})_{i,j}\right)\\ =&-\sup_{i\in \mathcal
      A}\sum_{x\in X_n}(M_{x,n})_{1,1}\log\left(\sum_{j\in \mathcal A}
    (M_{x,n})_{i,j}\right)\\ =& - \sum_{x\in X_n} (M_{x,n})_{1,1} \log
    \bigl( (M_{x,n})_{1,1} \bigr) \\ &- \sup_{i\in \mathcal A}\sum_{x
      \in X_n} (M_{x,n})_{1,1} \log \biggl(\frac{\sum_{j\in \mathcal
        A} (M_{x,n})_{i,j}}{(M_{x,n})_{1,1}} \biggr).
  \end{align*}
  The first term here is $ H_n(\beta,p)$. Since $\sum_{x\in
    X_n}(M_{x,n})_{1,1}=1$, we move this inside the
  $\log$ in the second term, and using the concavity of $\log$ we get
  \[
  L_n(\beta,p) \geq H_n(\beta,p) - \sup_{i\in \mathcal A} \log
  \left(\sum_{x\in X_n}\sum_{j\in \mathcal A} (M_{x,n})_{i,j} \right).
  \]
  Now recall that $(M_{x,n})_{i,j}$ counts, for any $a_1\cdots a_n$
  such that $\sum_{l=1}^n a_l\beta^{-l}=x$, the total measure of the
  words $b_1\cdots b_n$ for which
  \[
  \beta^nx_i + \sum_{l=1}^n (a_l-b_l)\beta^{n-l}=x_j.
  \]

  This can be rewritten as
  \begin{equation}\label{beq}
    \beta^nx_i + \beta^n x-\sum_{l=1}^n b_l\beta^{n-l}=x_j.
  \end{equation}

  In order to sum this over all $x\in X_n$ and $j\in \mathcal A$, we
  count for each $b_1\cdots b_n\in\{0,1\}^n$ the number of $x\in X_n$
  for which an equation of the form (\ref{beq}) is satisfied. This
  gives
  \begin{align*}
    \sum_{x\in X_n}\sum_{j\in \mathcal A} (M_{x,n})_{i,j} &= \sum_{x
      \in X_n} \sum_{j \in \mathcal A} \sum_{\substack{b_1\cdots b_n
        \in \{0,1\}^n \\ \text{\eqref{beq} holds}}} m_p [b_1 \cdots
      b_n] \\ &= \sum_{b_1\cdots b_n \in \{0,1\}^n} m_p [b_1 \cdots
      b_n] \cdot |X_n (i, b_1 \cdots b_n)|,
  \end{align*}
  where
  \[
  X_n (i, b_1 \cdots b_n) = \biggl\{ x \in X_n:\beta^n x_i + \biggl(
  \beta^n x - \sum_{l=1}^n b_l \beta^{n-l} \biggr) \in V_{\beta}
  \biggr\}.
  \]

  But now the separation arguments of Lemmma~\ref{separation} give
  that, for a fixed $i$ and $b_1 \cdots b_n$, sums of the form
  $\beta^n x - \sum_{l=1}^n b_l\beta^{n-l}$ are separated by at least
  $C_0(n)$ unless they are equal. This bounds the number of elements
  of $X_n (i, b_1 \cdots b_n)$. Indeed, all possible
  values of
  \[
  \beta^n x - \sum_{l=1}^n b_l\beta^{n-l}, \qquad x \in X_n (i, b_1
  \cdots b_n),
  \]
  are contained in the interval $\bigl[- \beta^n x_i - \frac{1}{\beta
      - 1}, - \beta^n x_i + \frac{1}{\beta - 1} \bigr]$ and they are
  separated by at least $C_0 (n)$.  Hence $\beta^n x - \sum_{l=1}^n
  b_l\beta^{n-l}$ may attain at most
  \[
  \frac{2}{\beta - 1} \frac{1}{C_0 (n)} + 1
  \]
  different values, that is
  \[
  |X_n (i, b_1 \cdots b_n)| \leq \frac{2}{\beta - 1} \frac{1}{C_0 (n)}
  + 1 = C(\beta) (n+1)^l + 1.
  \]
  Thus
  \begin{align}
    \sum_{x\in X_n}\sum_{j\in \mathcal A} (M_{x,n})_{i,j} & =
    \sum_{b_1\cdots b_n \in \{0,1\}^n} m_p [b_1 \cdots b_n] |X_n (i,
    b_1 \cdots b_n)| \nonumber \\ & \leq (C(\beta)(n+1)^l+1),
    \label{FinalIneq}
  \end{align}
  and so
  \[
  \frac{1}{n}L_n(\beta,p) \geq H
  (\beta,p)-\frac{2}{n}\log(C(\beta)(n+1)^l+1). \qedhere
  \]
\end{proof}

This completes the proof of both Theorem~\ref{thm1} and \ref{thm2}.

\subsection{A matrix form for the hyperbolic case}

We briefly comment on two alternative lower bounds which work for the
hyperbolic case and are much easier to work with. Let
\[
L_n'(\beta,p):=\sum_{a_1\cdots a_n\in\{0,1\}^n}m_p[a_1\cdots
  a_n]\log\left(\lVert (M_{a_1}\cdots M_{a_n}) \rVert \right).
\]
Here, the norm that we use is the row sum norm
\[
\lVert M \rVert = \sup_{i\in\mathcal A}\sum_{j\in\mathcal A}|M_{ij}|.
\]
$L_n'(\beta,p)$ differs from $L_n(\beta,p)$ in that the supremum over
$i\in \mathcal A$ happens inside the summation. Thus
$L_n'(\beta,p)\leq L_n(\beta,p)$.

\begin{lemma}
  When the set $V_{\beta}$ is finite, we have
  \[
  H(\beta,p)=\lim_{n\to\infty} \frac{1}{n}L_n'(\beta,p).
  \]
\end{lemma}

Proving this lemma completes the proof of Theorem~\ref{thm3}.

\begin{proof}
  The proof follows that of Proposition~\ref{prop:boundsonH} exactly
  to give
  \[
  L_n'(\beta,p)\geq H_n(\beta,p)-\log(\sum_{x\in X_n} \lVert M_{x,n}
  \rVert).
  \]

  But
  \begin{align*}
    -\log(\sum_{x\in X_n} \lVert M_{x,n} \rVert)&= -\log(\sum_{x\in
      X_n}\max_{i\in\mathcal A}\sum_{j\in\mathcal A}
    (M_{x,n})_{i,j})\\ &\geq - \log(\sum_{x\in X_n}\sum_{i\in\mathcal
      A}\sum_{j\in\mathcal A} (M_{x,n})_{i,j})\\ &\geq -\log(|\mathcal
    A|(C(\beta)+1),
  \end{align*}
  where the last line uses inequality~\eqref{FinalIneq} (with $l=0$),
  summing both sides over $i\in \mathcal A$. Then
  \[
  \frac{1}{n}L_n'(\beta,p)\geq
  \frac{1}{n}H_n(\beta,p)-\frac{1}{n}\log(|\mathcal A|(C(\beta)+1)),
  \]
  giving that $\frac{1}{n}L_n'(\beta,p)$ converges to $H(\beta,p)$ as
  required.
\end{proof}

Norms of random products of matrices are extremely well studied, and
so putting our lower bound for $H(\beta,p)$ in the above form may
yield useful computations.

We now describe another bound from below on $H(\beta,p)$, which is
computationally very simple, and which is sometimes sufficient to
conclude that the Hausdorff dimension of $\nu_{\beta,p}$ is 1.

\begin{prop} \label{prop:lowerbound}
  Suppose that $V_{\beta}$ is finite. Let $\lambda$ be the largest
  eigenvalue of the matrix $((1-p)M_0 + pM_1)$. Then
  \[
  -\log \lambda \leq H(\beta, p).
  \]
\end{prop}

\begin{proof}
  We use the norm $\lVert M \rVert_1 = \sum_{i,j} |M_{i,j}|$. For
  non-negative matrices $A$ and $B$, we have $\lVert A \rVert_1 + \lVert
  B \rVert_1 = \lVert A + B \rVert_1$.

  We have
  \begin{align*}
    \frac{1}{n} L_n'(\beta,p) &= -\frac{1}{n} \sum_{a_1 \cdots a_n}
     m_p[a_1\cdots a_n]\log \biggl(\lVert M_{a_1} \cdots M_{a_n}
      \rVert \biggr) \\
      &\geq -\frac{1}{n} \sum_{a_1 \cdots a_n}
    m_p[a_1\cdots a_n] \log \biggl( \lVert M_{a_1} \cdots M_{a_n}
      \rVert_1 \biggr) \\
      &\geq -\frac{1}{n} \log \biggl(
    \sum_{a_1 \cdots a_n} m_p[a_1\cdots a_n] \lVert M_{a_1} \cdots
      M_{a_n} \rVert_1 \biggr) \\
      &=- \frac{1}{n} \log
    \biggl( \biggl\lVert \sum_{a_1 \cdots a_n} m_p(a_1\cdots a_n)M_{a_1} \cdots
      M_{a_n} \biggr\rVert_1 \biggr) \\ &=
    -\frac{1}{n} \log \Bigl\lVert \Bigl( (1-p)M_0 + pM_1 \Bigr)^n
    \Bigr\rVert_1.
  \end{align*}
  By Proposition~\ref{prop:boundsonH}, $\frac{1}{n} L_n' (\beta,p)$ is a
  lower bound on $H (\beta,p)$ and since
  \[
  \lim_{n \to \infty} \frac{1}{n} \log \Bigl\lVert \Bigl( (1-p)M_0 +
    pM_1 \Bigr)^n \Bigr\rVert_1 = \log \lambda,
  \]
  we have $H(\beta,p) \geq - \log \lambda$.
\end{proof}

Since computing eigenvalues is extremely rapid, this approach is the
one that we use in practice for proving that $\dim_H(\nu_{\beta,p})=1$
for a variety of examples.

\subsection{Computational ideas and examples}

In this section we describe how to use
Proposition~\ref{prop:boundsonH} to get explicit bounds on
$H(\beta)=H(\beta,1/2)$ and hence on $\dimH \nu_\beta$ for specific
examples. For the remainder of the article we concern ourselves only
with the case of unbiased Bernoulli convolutions, and no longer
include $p$ as a variable.

Suppose $\beta$ is hyperbolic. Then one easily writes a computer
program which finds the (finite) graph $\mathcal{G}$ and the matrices
$M_0$ and $M_1$. By Proposition~\ref{prop:boundsonH}, we have
\[
\frac{1}{n} L_n (\beta) \leq H(\beta) \leq \frac{1}{n} H_n (\beta).
\]
Expressed as a bound on $\dimH \nu_\beta$, it says
\[
 \min \biggl\{ 1, \frac{1}{n} \frac{L_n (\beta)}{\log \beta} \biggr\}
 \leq \dimH \nu_\beta \leq \min \biggl\{1,\frac{1}{n} \frac{H_n
   (\beta)}{\log \beta}\biggr\}.
\]

Given an $n$ one can, with a computer, calculate numerically the above
lower and upper bounds on $H (\beta)$ and $\dimH
\nu_\beta$. Unfortunately, the convergence is quite slow, and the
computational complexity is high, since evaluating $L_n (\beta)$ and
$H_n (\beta)$ involves summing over $2^n$ different sequences.

There is a way to somewhat improve the convergence by pruning the
graph $\mathcal{G}$. Call a vertex $x$ redundant if there is no path
to $\{0\}$ from $x$ along edges in the graph. Clearly, a vertex is
redundant, if and only if all edges from the vertex lead to redundant
vertices. We remove all redundant vertices from $\mathcal{G}$ and get
a new graph which we denote by $\mathcal{G}'$. Using instead this
pruned graph to define $\tilde{L}_n (\beta)$ in the same
way as the definition of $L_n' (\beta)$, the above
bounds on $H(\beta)$ and $\dimH \nu_\beta$ hold with $L_n' (\beta)$
replaced by $\tilde L_n (\beta)$.

\begin{example} \label{ex:bounds}
  To illustrate the above, we let $\beta$ be
  the largest root of the equation
  \[
  \beta^4-\beta^3-\beta^2+\beta-1=0.
  \]
  Here $\beta \approx 1.5129$ has one other conjugate $\beta^{(2)}$
  larger than one in modulus, $\beta^{(2)}\approx -1.1787$. $\beta$
  also has two conjugates less than one in modulus, both of which are
  complex. $\mathcal{G}$ consists of 67 vertices and $\mathcal{G}'$
  consists of 21 vertices.

  Using the graph $\mathcal{G}$ and $n=9$, we find that
  \[
  \frac{1}{n} \frac{L_n' (\beta)}{\log \beta} = 0.77199 \leq \frac{H
    (\beta)}{ \log \beta} \leq 1.5763 = \frac{1}{n} \frac{H_n
    (\beta)}{\log \beta}.
  \]
  Using instead the pruned graph $\mathcal{G}'$ and $n=9$, we find that
  \[
  \frac{1}{n} \frac{\tilde L_n (\beta)}{\log \beta} = 1.0006 \leq
  \frac{H(\beta)}{\log \beta}.
  \]
  We conclude that $\dimH(\nu_{\beta})=1$.  We remark that this result
  does not follow from the aforementioned work of Breuillard and Varju
  \cite{VarjuBreuillard1}, since in this
  example $$\frac{0.44\min\{\log 2, \log M_\beta\}}{\log \beta}\approx
  0.6146<1.$$
\end{example}

As is illustrated in the above example, even if the upper and lower
bounds are far apart, they can still be useful to prove that the
Hausdorff dimension is 1. For the number $\beta$ in the example it is
sufficient to take $n=9$ in order to prove that the dimension is
1. For some other numbers, one needs to take larger values of $n$,
resulting in very long computation times.

\subsubsection{Using Proposition~\ref{prop:lowerbound}}

We now give some examples to show the advantage of using
Proposition~\ref{prop:lowerbound}. The key advantage of this
proposition lies in the fact that eigenvalues are numerically quick to
compute.

\begin{example}
  We take $\beta$ as in
  Example~\ref{ex:bounds}. Proposition~\ref{prop:lowerbound} gives
  \[
  \frac{H(\beta)}{\log \beta} \geq 1.3867
  \]
  and hence $\dimH \nu_\beta = 1$.
\end{example}

The lower bound for $H(\beta,p)$ given in
Proposition~\ref{prop:lowerbound} is not tight. By looking at
Bernoulli convolutions associated with Pisot numbers one can see how
far off the true value it is for some examples.
\begin{example}
  Let $\beta$ be the Golden ratio. Alexander and Zagier showed that
  that $\dimH \nu_\beta = 0.995570\ldots$ \cite{AlexanderZagier}.
  Proposition~\ref{prop:lowerbound} gives
  \[
  \frac{H(\beta)}{\log \beta} \geq 0.9924
  \]
  and hence $\dimH \nu_\beta \geq 0.9924$.
\end{example}

Proposition~\ref{prop:lowerbound} also gives new information for Pisot
numbers. The fact that the dimension of the Bernoulli convolution in
the previous example is known so accurately is due to special
properties of the Golden ratio. Outside of a special class of Pisot
numbers known as multinacci numbers, there are no examples of Pisot
numbers for which the Hausdorff dimension of $\nu_{\beta}$ was known
to three decimal places before the present work, see
\cite{HareSidorov2}. See also \cite{FengRademacher}, in which Feng
calculated the Hausdorff dimension with high precision for multinacci
numbers.

\begin{example}
  Let $\beta$ be the Pisot number given by
  $\beta^3-\beta-1=0$. Since $\beta$ is a Pisot number, we have
  $\dimH \nu_\beta < 1$. Proposition~\ref{prop:lowerbound} gives
  \[
  \frac{H(\beta)}{\log \beta} \geq 0.99999.
  \]
  Hence $0.99999 \leq \dimH \nu_\beta < 1$ and we have obtained the
  Hausdorff dimension of $\nu_{\beta}$ to five decimal places.
\end{example}

Finally, we apply our methods to the study of hyperbolic $\beta$ of
degree $4$ and $5$.
\begin{example} \label{ex:degree4}
Let $\beta$ satisfy
\[
a_4 \beta^4+ a_3\beta^3+\cdots +a_0=0
\]
with each $a_i\in\{-1,0,1\}$. Suppose that $\beta$ is hyperbolic. Then
either $\beta$ is Pisot and $\nu_{\beta}$ has Hausdorff dimension less
than one, or $\beta$ is not Pisot and $\nu_{\beta}$ has dimension
one. The computations are given in Table~\ref{tab:degree2-4}, which
shows all hyperbolic $\beta$ that are roots of a
$\{-1,0,1\}$-polynomial of degree 2, 3 or 4.
\end{example}

\begin{table}
\begin{center}
\tiny
\begin{tabular}{r|c|c|c|r}
Polynomial & $\beta$ & type & $\frac{-\log \lambda}{\log \beta}$ & size of $\mathcal{G}'$  \\
\hline
$x^2 - x - 1$ & 1.6180 & Pisot & 0.99240 & 5 \\
\hline
$x^3 - x^2 - x - 1$ & 1.8393 & Pisot & 0.96422 & 7 \\
$x^3 - x^2     - 1$ & 1.4656 & Pisot & 0.99912 & 49 \\
$x^3       - x - 1$ & 1.3247 & Pisot & 0.99999 & 179\\
\hline
$x^4 - x^3 - x^2 - x - 1$ & 1.9276 & Pisot & 0.97333 & 9 \\
$x^4 - x^3 - x^2 + x - 1$ & 1.5129 & not Pisot & 1.38670 & 21 \\
$x^4 - x^3           - 1$ & 1.3803 & Pisot & 0.99999 & 1253 \\
$x^4 - x^3 + x^2 - x - 1$ & 1.2906 & not Pisot & 2.50349 & 9\\
$x^4       - x^2     - 1$ & 1.2720 & not Pisot & 1.98480 & 25\\
$x^4             - x - 1$ & 1.2207 & not Pisot & 1.61576 & 1693\\
$x^4 + x^3 - x^2 - x - 1$ & 1.1787 & not Pisot & 3.49147 & 21
\end{tabular}
\caption{Lower bounds for all hyperbolic $\beta$ of degree 2, 3 and 4.}
\label{tab:degree2-4}
\end{center}
\end{table}

We also attempted to compute the dimension of all hyperbolic $\beta$
that are roots of a $\{-1,0,1\}$-polynomial of degree 5. As can be
seen in Table~\ref{tab:degree5}, which shows all such numbers, some
$\beta$ give rise to very large graphs for which the computation is
not feasible on a standard computer. An alternative approach to this
case is discussed in the comments section.

\begin{table}
\begin{center}
\tiny
\begin{tabular}{r|c|c|c|r}
Polynomial & $\beta$ & type & $\frac{\log 2 - \log \lambda}{\log \beta}$ & size of $\mathcal{G}'$ \\
\hline
$x^5 - x^4 - x^3 - x^2 - x - 1$ & 1.9659 & Pisot & 0.98357 & 11 \\
$x^5 - x^4 - x^3 - x^2     - 1$ & 1.8885 & Pisot & 0.98227 & 739 \\
$x^5 - x^4 - x^3 - x^2     + 1$ & 1.7785 & Pisot & 0.99576  & 947 \\
$x^5 - x^4 - x^3 - x^2 + x - 1$ & 1.7924 & not Pisot & 1.12741  & 13 \\
$x^5 - x^4 - x^3       - x - 1$ & 1.8124 & Pisot & 0.98243 & 349 \\
$x^5 - x^4 - x^3       - x + 1$ & 1.6804 & not Pisot & 1.17467 & 139 \\
$x^5 - x^4 - x^3           - 1$ & 1.7049 & Pisot & 0.99304 & 339 \\
$x^5 - x^4 - x^3       + x - 1$ & 1.5499 & not Pisot & 1.1971 & 1931 \\
$x^5 - x^4 - x^3 + x^2     - 1$ & 1.4432 & Pisot & 0.99998 & 5387 \\
$x^5 - x^4       - x^2 - x - 1$ & 1.6851 & not Pisot & 1.1072 & 2055 \\
$x^5 - x^4       - x^2 - x + 1$ & 1.5262 & not Pisot & 1.4420 & 139 \\
$x^5 - x^4       - x^2     - 1$ & 1.5702 & Pisot & 0.99986 & 841 \\
$x^5 - x^4       - x^2 + x - 1$ & 1.4036 & not Pisot & 1.3664 & 2041 \\
$x^5 - x^4             - x - 1$ & 1.4971 & not Pisot & 1.4216 & 57 \\
$x^5 - x^4       + x^2 - x - 1$ & 1.2628 & not Pisot & 2.4946 & 131 \\
$x^5 - x^4 + x^3 - x^2 - x - 1$ & 1.4076 & not Pisot & 1.9447 & 11 \\
$x^5 - x^4 + x^3 - x^2     - 1$ & 1.2499 & not Pisot & 1.8291 & 1877 \\
$x^5 - x^4 + x^3       - x - 1$ & 1.2083 & not Pisot & 3.3882 & 11 \\
$x^5       - x^3 - x^2 - x - 1$ & 1.5342 & Pisot & 0.99983 & 2635 \\
$x^5       - x^3 - x^2 - x + 1$ & 1.3690 & not Pisot & 1.8252 & 1119 \\
$x^5       - x^3 - x^2     - 1$ & 1.4291 & not Pisot & 1.6106 & 43 \\
$x^5       - x^3 - x^2 + x - 1$ & 1.2828 & not Pisot & 2.5554 & 13 \\
$x^5       - x^3           - 1$ & 1.2365 & not Pisot & ? & 45563 \\
$x^5       - x^3 + x^2 - x - 1$ & 1.2000 & not Pisot & ? & ? \\
$x^5             - x^2     - 1$ & 1.1939 & not Pisot & ? & 15211 \\
$x^5                   - x - 1$ & 1.1673 & not Pisot & ? & ? \\
$x^5       + x^3 - x^2 - x - 1$ & 1.1436 & not Pisot & 4.3718 & 97 \\
$x^5 + x^4 - x^3 - x^2     - 1$ & 1.1595 & not Pisot & 4.3005 & 13 \\
$x^5 + x^4 - x^3       - x - 1$ & 1.1408 & not Pisot & 4.6298 & 139 \\
$x^5 + x^4       - x^2 - x - 1$ & 1.1237 & not Pisot & ? & 32179 \\
\end{tabular}
\caption{Lower bounds for all hyperbolic $\beta$ of degree 5.}
\label{tab:degree5}
\end{center}
\end{table}

\section{Further Comments and Questions}

\begin{enumerate}
\item As can be seen from Table~\ref{tab:degree5}, when $\beta$ has
  Galois conjugates close to $1$ in modulus the graph $\mathcal G$ can
  be very large. In these cases, calculating the graph $\mathcal G$
  may not be the most efficient way of proving that $\nu_{\beta}$ has
  dimension 1. In a follow up article we show how one can perform
  counting estimates broadly similar to those of
  \cite{HareSidorov1,HareSidorov2} on a higher dimensional self-affine
  set with contraction ratios equal to the Galois conjugates of
  $\beta$. These estimates often yield that $\dimH(\nu_{\beta})=1$,
  and work even in the case of non-hyperbolic $\beta$.
\item A short argument of Mercat (personal communication) shows that
  $H(\beta)\leq \log(\beta)$ whenever $\beta$ is a Salem
  number. Therefore it will hold that $L_n(\beta)<\log(\beta)$ for all
  $n\in\mathbb N$ and so our finite time approximation methods will
  not be able to show that $\dim_H(\nu_{\beta})=1$ for $\beta$ Salem.
\end{enumerate}

\section*{Acknowledgments}
We would like to thank Pierre Arnoux and Paul Mercat for many helpful
conversations.

The majority of this work took place at the Centre International de
Recontres Math\'{e}matiques in Luminy during a Research in Pairs
program. We would like to thank the CIRM and the Jean-Morlet semester
``Tiling and discrete geometry' for financial and logistical
support. D.-J.~Feng, T.~Kempton and T.~Persson are also grateful to
Institut Mittag-Leffler in Djursholm and the semester ``Fractal
geometry and dynamics'' during which parts of this work were
done. D.-J.~Feng was also partially supported by the HKRGC GRF grants
(projects CUHK2130445, CUHK2130546).

\bibliographystyle{plain}
\bibliography{RIP}

\bigskip

S.~Akiyama, Institute of Mathematics, University of Tsukuba, 1-1-1 Tennodai, Tsukuba, Ibaraki, 305-8571, Japan\newline\nopagebreak
\textit{e-mail address:}~\texttt{akiyama@math.tsukuba.ac.jp}

D.-J.~Feng, Department of Mathematics, Room 211, Lady Shaw Building, The Chinese University of Hong Kong, Shatin, N.~T., Hong Kong\newline\nopagebreak
\textit{e-mail address:}~\texttt{djfeng@math.cuhk.edu.hk}

T.~Kempton, School of Mathematics, The University of Manchester, Manchester, M13 9PL,
United Kingdom\newline\nopagebreak
\textit{e-mail address:}~\texttt{thomas.kempton@manchester.ac.uk}

T.~Persson, Centre for Mathematical Sciences, Lund University, Box 118, 221~00 Lund, Sweden\newline\nopagebreak
\textit{e-mail address:}~\texttt{tomasp@maths.lth.se}

\end{document}